\documentclass[leqno,12pt,draft]{amsart}

\usepackage{amssymb}
\usepackage{amsmath}
\usepackage{enumerate}
\usepackage{amsfonts}
\usepackage{color}

\numberwithin{equation}{section}

\newtheorem{remark}[equation]{Remark}
\newtheorem{lema}[equation]{Lemma}
\newtheorem{teo}[equation]{Theorem}

\newtheorem{coro}[equation]{Corollary}
\begin{document}
\title[Riesz transforms]{On dimension-free and  potential-free estimates\\ for Riesz transforms
\\
associated with Schr\"odinger operators}
\author{Jacek Dziuba\'{n}ski }
\address{Instytut Matematyczny, Uniwersytet Wroc\l awski, pl. Grunwaldzki 2/4, 50-384 Wroc\l aw, Poland} \email{jdziuban@math.uni.wroc.pl }
\subjclass[2020]{42B20, 35J10 (primary);  47D08, 47G10 (secondary)}
\keywords{Riesz transforms, Schr\"odinger operators}
\thanks{}

\maketitle

\begin{abstract}
 Let  $L=-\Delta + V(x)$ be a Schr\"odinger operator on $\mathbb R^d$,   where $V(x)\geq 0$, $V\in L^1_{\rm loc} (\mathbb R^d)$. We give a short proof of dimension free $L^p(\mathbb R^d)$ estimates, $1<p\leq 2$, for the vector of the Riesz transforms $$\big(\frac{\partial}{\partial x_1}L^{-1/2}, \frac{\partial}{\partial x_2}L^{-1/2},\dots,\frac{\partial}{\partial x_d}L^{-1/2}\Big).$$ The constant in the estimates does not depend on  the potential $V$. We simultaneously provide a short proof of the weak type $(1,1)$ estimates for $\frac{\partial}{\partial x_j}L^{-1/2}$.
\end{abstract}


\section{Introduction and statement of the result}

Let $R_jf= \frac{\partial}{\partial x_j} (-\Delta)^{-1/2}f$ denote the Riesz transform in $\mathbb R^d$. It was proved in Stein \cite{Stein1983} that for any $1<p<\infty$, there is a constant $C_p$ independent of $d$ such that
\begin{equation}
    \label{eq:Stein}
     \Big\| \Big(\sum_{j=1}^d | R_j f|^2\Big)^{1/2} \Big\|_{L^p(\mathbb R^d)}\leq C_p\| f\|_{L^p(\mathbb R^d)}.
\end{equation}
It is known that for $1<p\leq 2$, the best constant in \eqref{eq:Stein} can be estimated from above by $C(p-1)^{-1}$ (see e.g. \cite{B}, \cite{IM}).

In this paper we consider the Riesz transforms $\mathcal R_j=\frac{\partial}{\partial x_j} L^{-1/2}$, where  $L=-\Delta+V$ is a Schr\"odinger operator in $\mathbb R^d$ with a non-negative potential  $V\in L^1_{\rm loc}(\mathbb R^d)$.  Our goal is to prove the following theorem which, according to the author's knowledge, is new.
\begin{teo}\label{teo} Suppose $1<p\leq 2$ and $d\geq 3$.
    Let $C_p$ be a constant for which the dimension-free $L^p$ estimates \eqref{eq:Stein} for the classical Riesz transforms hold. Then
    \begin{equation}\label{free} \Big\| \Big(\sum_{j=1}^d |\mathcal R_j f|^2\Big)^{1/2} \Big\|_{L^p(\mathbb R^d)}\leq 2^{(2-p)/p} C_p\| f\|_{L^p(\mathbb R^d)}.
    \end{equation}
\end{teo}

Estimates   of Riesz transforms associated with Schr\"odinger operators on  the Euclidean spaces and Riemannian manifolds were considered by many authors. For some  results we refer  the reader to \cite{AO, AB, De, DV, DOY, HRST, K, KW, L-P, Sh, Sikora2004, T, UZ2009} and references therein.  However, all the papers which were  dealing with  dimension-free $L^p$ bounds  for the Riesz transforms associated with Schr\"odinger operators  required some additional assumptions on the non-negative potential $V$ to conclude \eqref{free} with a constant $C_p'$ independent of $d$.
For example in \cite{UZ2009} the authors obtained \eqref{free} for potentials $V$ being
polynomials satisfying so called Fefferman condition, however,  the constant $C_p'$  depended on the degree of $V$.   In the case of the harmonic oscillator   $\mathbf L=-\Delta+\|x\|^2$ in $\mathbb R^d$, the $L^p$-bounds independent of the dimension $d$ for the  vector of the Riesz transforms $(\frac{\partial}{\partial x_1}\mathbf L^{-1/2},x_1\mathbf L^{-1/2}, \frac{\partial}{\partial x_2}\mathbf L^{-1/2},x_2\mathbf L^{-1/2},\dots, \frac{\partial}{\partial x_d}\mathbf L^{-1/2},x_d\mathbf L^{-1/2})$ were proved in \cite{HRST} and \cite{L-P}, see also \cite{DV}. Let us note that these results can be also achieved  by applying the dimension-free $L^p$ estimates for the Riesz transforms on Heisenberg groups proved in \cite{CMZ} and the transference method of Coifman and Weiss \cite{CW}.

 The proof of Theorem \ref{teo} is based on the factorization $\mathcal R_j=R_j(-\Delta)^{1/2}L^{-1/2}$, the inequality \eqref{eq:Stein}, and  $L^p$ bounds for $(-\Delta)^{1/2}L^{-1/2}$ (see Lemma \ref{lemma1}). The bounds are obtain in a short way by utilizing the perturbation (Duhamel) formula.

Finally we want to remark that the range $1<p\leq 2$ is optimal for the $L^p$-boundedness of the Riesz transforms associated with $L=-\Delta+V$,   $V\geq 0$ (see Section \ref{optimal}).

 \section{Proof of Theorem \ref{teo}}

Let $h_t(x-y)$ and $k_t(x,y)$ denote the integral kernels of the heat semigroup $H_t=e^{t\Delta}$ and the Schr\"odinger semigroup $K_t=e^{-tL}$ respectively. It is well-known and follows e.g.from the Feynman-Kac formula that
 \begin{equation}
     \label{eq11}
     0\leq k_t(x,y)\leq h_t(x-y)=(4\pi t)^{-d/2} \exp(-\|x-y\|^2/4t).
 \end{equation}
Negative  powers of the generators are defined by:
\begin{equation}\begin{split}\nonumber
   & (-\Delta)^{-1} f(x)= \int_0^\infty H_tf(x)\, dt
   =\int_0^\infty \int_{\mathbb R^d} h_t(x-y) f(y)\, dy\, dt=:\int_{\mathbb R^d} \Gamma_0(x-y)f(y)\, dy, \\
   & L^{-1}f(x)=\int_0^\infty K_tf(x)\, dt =\int_0^\infty \int_{\mathbb R^d} k_t(x,y)f(y)\, dy \, dt=: \int_{\mathbb R^d} \Gamma(x,y)f(y)\, dy,\\
    & (-\Delta)^{-1\slash 2} f= c_1 \int_0^\infty H_tf\,\frac{ dt}{\sqrt{t}}= c_1 \int_0^\infty \int_{\mathbb R^d} h_t(x-y)f(y)\, dy \frac{dt}{\sqrt{t}}
   =:\int_{\mathbb R^d} \widetilde\Gamma_0(x-y)f(y)\, dy ,  \\
   & L^{-1\slash 2} f= c_1\int_0^\infty K_tf\,\frac{ dt}{\sqrt{t}}=c_1\int_0^\infty \int_{\mathbb R^d} k_t(x-y)f(y)\, dy \frac{dt}{\sqrt{t}}
   =:\int_{\mathbb R^d} \widetilde \Gamma (x,y)f(y)\, dy,\\
\end{split}\end{equation}
where $ c_1=\Gamma (1\slash 2)^{-1}=\pi^{-1/2}.$
 Clearly,

 \begin{equation}\begin{split}\label{eq22.1}
  0\leq \tilde \Gamma (x,y)\leq \tilde \Gamma_0(x-y)=c_d|x-y|^{1-d}, \ \ \ 0\leq  \Gamma(x,y)\leq \Gamma_0(x-y),
\end{split}\end{equation}
\begin{equation}
    \label{composition}
    \int_{\mathbb R^d} \tilde \Gamma(x,z)\tilde\Gamma (z,y)\, dz =\Gamma (x,y).
\end{equation}
\begin{lema}\label{int1}
    $$\int_{\mathbb R^d} V(z)\Gamma (z,y)\, dz\leq 1. $$
\end{lema}
\begin{proof} The lemma is well known (see e.g., \cite{AB}, \cite{UZ2009}). Its easy proof is a consequence of the  perturbation (Duhamel) formula:
    \begin{equation}\label{eq:pert}
        h_t(x-y)=k_t(x,y)+\int_0^t \int_{\mathbb R^d} h_{t-s}(x-z)V(z)k_{s} (z,y)\, dz\, ds.
    \end{equation}
   Indeed,  integrating \eqref{eq:pert} with respect to $dx$ and using \eqref{eq11}, we get
    \begin{equation}
        1 = \int_{\mathbb R^d}k_t(x,y)\, dx+  \int_{\mathbb R^d} \int_0^t V(z) k_s (z,y)\,  ds\,  dz\geq  \int_{\mathbb R^d} V(z) \int_0^t k_s (z,y)\,  ds\,  dz .
    \end{equation}
    Letting $t\to\infty$ and using the monotone convergence, we obtain the lemma.
\end{proof}

 For a reasonable function $f$ (for example for $f$ being in the domain of the relevant quadratic form (see \eqref{eq:form}),  let
 \begin{equation}\begin{split}\nonumber
   &(-\Delta)^{1\slash 2} f=c_2\int_0^\infty  (H_tf-f)\frac{dt}{t^{3\slash 2}}, \\
   &L^{1\slash 2} f= c_2\int_0^\infty  (K_tf-f)\frac{dt}{t^{3\slash 2}},\\
\end{split}\end{equation}
 where  $  c_2=\Gamma (-1\slash 2)^{-1}=-(2\sqrt{\pi})^{-1}$.
\begin{lema}\label{lemma1}
 \begin{equation}
     \label{eq:L2}
     \|(-\Delta)^{1/2} L^{-1/2} f\|_{L^2(\mathbb R^d)} \leq \| f\|_{L^2(\mathbb R^d)}, 
 \end{equation}
   \begin{equation}\label{eq2.51}
   \| (-\Delta)^{1\slash 2}L^{-1\slash 2}f\|_{L^1(\mathbb R^d)}\leq 2\| f\|_{L^1(\mathbb R^d)}.
   \end{equation}
\end{lema}
\begin{proof} The proof of well known estimates  \eqref{eq:L2} follows from quadratic form representations  of $-\Delta$ and $L$:
\begin{equation}\begin{split}\label{eq:form}
& \| (-\Delta)^{1/2}  f\|_{L^2(\mathbb R^d)}^2 =\int_{\mathbb R^d} (-\Delta)^{1/2}f \cdot \overline{ (-\Delta)^{1/2}f }
=\int_{\mathbb R^d} \nabla f\cdot \overline{\nabla f}\\
&\leq \int_{\mathbb R^d} \nabla f\cdot \overline{\nabla f} + V^{1/2} f\cdot \overline{V^{1/2}f}=\int_{\mathbb R^d} (L^{1/2} f )\cdot \overline{(L^{1/2} f)}=\| L^{1/2} f\|_{L^2(\mathbb R^d)}^2.
\end{split}\end{equation}

Before we  turn to prove \eqref{eq2.51} a comment is in order. The $L^1$-bounds for the operators $(-\Delta)^{1/2}L^{-1/2}$ and $L^{1/2}(-\Delta)^{-1/2}$ were studied in \cite{DZ} under the additional assumption that $V$ is Green bounded, that is,
$$\sup_{x\in\mathbb R^d} \int_{\mathbb R^d} |x-y|^{2-d}V(y)\, dy <\infty. $$
However,  for  the proof of the uniform bound \eqref{eq2.51}, it suffices to utilize Lemma \ref{int1}. Indeed, using the definitions of fractional powers of operators and   the perturbation formula \eqref{eq:pert},  we get
\begin{equation}\label{eq2.8}
  \begin{split}
   &(-\Delta)^{1\slash 2} L^{-1\slash 2} f (x) = c_2  \int_0^\infty (H_t - I) L^{-1\slash 2} f(x)\frac{dt}{t^{3\slash 2}}\\
   & = c_2 \int_0^\infty (H_t - K_t) L^{-1\slash 2} f(x)\frac{dt}{t^{3\slash 2}} +
    c_2\int_0^\infty (K_t - I) L^{-1\slash 2} f(x)\frac{dt}{t^{3\slash 2}}\\
    & = c_2\int_0^\infty \int_0^t \int_{\mathbb R^d}\int_{\mathbb R^d}  h_{t-s}(x-z)V(z)k_s(z,y) (L^{-1\slash 2}f)(y) \, dy\, dz\, ds \frac{dt}{t^{3\slash 2}}  +f(x). \\
  \end{split}
 \end{equation}
 Consider the integral kernel $W(x,u)$ of the operator
 $$f\mapsto \int_0^\infty \int_0^t \int_{\mathbb R^d}\int_{\mathbb R^d}  h_{t-s}(x-z)V(z)k_s(z,y) (L^{-1\slash 2}f)(y) \, dy\, dz\, ds \frac{dt}{t^{3\slash 2}},$$ that is,
 \begin{equation}\label{eq:W}
     W(x,u)=\int_0^\infty \int_0^t \int_{\mathbb R^d}\int_{\mathbb R^d} h_{t-s}(x-z)V(z)k_s(z,y) \tilde \Gamma (y,u) \, dy\, dz\, ds \frac{dt}{t^{3\slash 2}}.
 \end{equation}
 Clearly $0\leq W(x,u)$, since the integrand is non-negative. Integration of  \eqref{eq:W}  with respect to $dx$ and the use of the Fubini theorem  and \eqref{composition}  lead to
 \begin{equation}\begin{split}\label{eq2.88}
\int_{\mathbb R^d} W(x,u)\, dx & = \int_0^\infty \int_0^t \int_{\mathbb R^d}\int_{\mathbb R^d} V(z)k_s(z,y) \tilde \Gamma (y,u) \, dy\, dz\, ds \frac{dt}{t^{3\slash 2}}\\
&=  2 \int_0^\infty  \int_{\mathbb R^d}\int_{\mathbb R^d}  V(z)k_s(z,y) \widetilde\Gamma (y,u) \, dy\, dz\,
\frac{ds}{\sqrt{s}} \\
&= 2c_1^{-1} \int_{\mathbb R^d}\int_{\mathbb R^d}  V(z)\widetilde\Gamma (z,y) \widetilde\Gamma (y,u)\, dy\, dz \\
&= 2c_1^{-1}\int_{\mathbb R^d} V(z) \Gamma (z,u)  dz\\
& \leq  2c_1^{-1} ,\\
 \end{split}
 \end{equation}
  where  in the last inequality we have used Lemma \ref{int1}. Taking  \eqref{eq2.8} together with \eqref{eq2.88}, we obtain \eqref{eq2.51}.
 \end{proof}

 The following corollary is a direct consequence  of Lemma \ref{lemma1} and the Riesz--Thorin interpolation theorem.
 \begin{coro}\label{coro}
     \label{Lp} For $1\leq p\leq 2$, we have
     $$ \|(- \Delta)^{1/2}L^{-1/2} f\|_{L^p(\mathbb R^d)}\leq 2^{(2-p)/p}\| f\|_{L^p(\mathbb R^d)}.$$
 \end{coro}

\begin{proof}[Proof of Theorem \ref{teo}] Observe that
\begin{equation}\label{eq:Riesz-Riesz}
\mathcal R_j=R_j ((-\Delta)^{1/2} L^{-1/2}).
\end{equation}
 Thus the theorem follows from  Corollary \ref{coro} and \eqref{eq:Stein}.
\end{proof}

\begin{remark}
    Let us note that from \eqref{eq2.51} and \eqref{eq:Riesz-Riesz} we easily recover the known results about the weak type $(1,1)$ estimates for $\mathcal R_j$ (see e.g. \cite{Sikora2004}).
\end{remark}

 \begin{remark}
     A careful reading of the proofs of Theorems 2.4 and 2.5 of \cite{KW} gives for any fixed $1<p\leq 2$ the bound
     \begin{equation}\label{eqKW}   \| V^{1/2}L^{-1/2} f\|_{L^p(\mathbb R^d)}\leq C_p'\| f\|_{L^p(\mathbb R^d)}\end{equation}
     with $C_p'$ independent of $d$ and $V\geq 0$. Thus, Theorem \ref{teo} together with \eqref{eqKW} give dimension free estimate for the vector $(V^{1/2}L^{-1/2},\frac{\partial}{\partial x_1}L^{-1/2},\dots,\frac{\partial}{\partial x_d}L^{-1/2})$ of the Riesz transforms associated with $L=-\Delta +V$ with $V\geq 0  $, $V\in L^1_{\rm loc}(\mathbb R^d)$.
 \end{remark}

\begin{remark}
    The method described above allows us to obtain, with the same short proof, weak type $(1.1)$ and $L^p$ boundedness, $1<p\leq 2$, of  Riesz transforms associated with Schr\"odinger operators with non-negative potentials  in some other  settings. Examples are the Dunkl-Schr\"odinger-Riesz transforms $T_j\mathcal L^{-1/2}$, where $T_j$ are the Dunkl operators associated with a root system $R$ and a multiplicity function $k\ge 0$ on $\mathbb R^d$, $\mathcal L=-\sum_{j=1}^d T_j^2+V$, $V\geq 0$, $V\in L^1_{\rm loc}(dw)$. Here $dw$ is the associated measure. Such bounds for $T_j\mathcal L^{-1/2}$ were obtained with  different proofs  in \cite{Amri}.
 \end{remark}

\section{Counterexamples}\label{optimal}
In this section we illustrate that for $p\geq 1$, $p\notin (1,2]$, the $L^p$ bounds for $\frac{\partial}{\partial x_j} L^{-1/2}$ and for $V^{1/2}L^{-1/2}$ fail, unless we impose some additional assumptions on $V$.

{\bf Counterexample 1.} { Let $2<p<\infty$ and $d\geq 3$. There is $V\in L^1_{\rm loc} (\mathbb R^d)$, $V\geq 0$,  such that $\frac{\partial}{\partial x_1}L^{-1/2}$ is not bounded on $L^p(\mathbb R^d)$. }

The construction is a straightforward  adaptation of that of Shen \cite[Section 7]{Sh}. For the reader's convenience we provide details. Write  $\mathbb R^d=\mathbb R^2\times \mathbb R^{d-2}\ni (x_1,x_2,x')=x$, $x_1,x_2\in\mathbb R$, $x'\in\mathbb R^{d-2}$. For a fixed   $0<\varepsilon<1-2/p$, we set
$$ V(x)=V(x_1,x_2,x')=(x_1^2+x_2^2)^{(\varepsilon -2)/2}, \quad v(x_1,x_2,x')=\sum_{m=0}^\infty \frac{(x_1^2+x_2^2)^{\varepsilon m/2}}{\varepsilon^{2m}(m!)^2}.  $$
Then $  (-\Delta+V)v=0$. Let $\phi\in C_c^{\infty}(\mathbb R^d)$, $\phi (x)=1$ for $|x|\leq 1$. Put $u(x)=\phi(x)v(x)$, $g=(-\Delta +V)u$. Then $g=-v\Delta \phi -2 \nabla v\cdot \nabla \phi\in C_c^\infty(\mathbb R^d)$ and $\frac{\partial}{\partial x_1} u\notin L^p(\mathbb R^d)$, because
$$ \Big|\frac{\partial}{\partial x_1} u(x_1,x_2,x')\Big|\sim  |x_1|(x_1^2+x_2^2)^{\varepsilon/2-1} \quad \text{near zero}.$$
If $\frac{\partial}{\partial x_1}L^{-1/2}$ were bounded on $L^p(\mathbb R^d)$, then, by \eqref{eq22.1},
\begin{equation*}
    \begin{split}
        \Big\| \frac{\partial}{\partial x_1} u\Big\|_{L^p}^p =\Big\| \frac{\partial}{\partial x_1} L^{-1/2} L^{-1/2} g\Big\|_{L^p} \leq C\| L^{-1/2} g\|_{L^p}^p \leq C_g\int_{\mathbb R^d}(1+|x|)^{(1-d)p} \,dx <\infty
    \end{split}
\end{equation*}
and we arrive at a contradiction.

{\bf Counterexample 2.} Suppose $2<p<\infty$ and $d\geq 3$. There is $V\in L^1_{\rm loc} (\mathbb R^d)$, $V\geq 0$,  such that $V^{1/2}L^{-1/2}$ is not bounded on $L^p(\mathbb R^d)$.

Indeed, let $V(x)=V(x_1,x')=|x_1|^{-2/p}\chi_{B(0,1)}(x)$, $x=(x_1,x')\in \mathbb R\times \mathbb R^{d-1}$, $B(0,1)=\{x\in\mathbb R^d: |x|<1\}$. Then the integral kernel  $k_t(x,y)$ of the semigroup $e^{-tL}$, $L=-\Delta+V$, satisfies the upper and lower  Gaussian bounds, that is, there are constants  $C,c>0$ such that
\begin{equation}
    \label{Gaussian-bounds}
    C^{-1} h_{ct} (x-y)\leq k_t(x,y)\leq h_t(x-y)
\end{equation}
 (see \cite[Section 5]{BDS}). Consequently,
\begin{equation}\label{similar} \tilde \Gamma(x,y)\sim \tilde \Gamma_0(x-y)=c_d|x-y|^{1-d}.
\end{equation}
Let $f(x)=\chi_{B(0,1)}(x)$. Then $V^{1/2}L^{-1/2}f\notin L^p(\mathbb R^d)$, because, thanks to \eqref{similar},
\begin{equation*}
    \begin{split}
        V^{1/2}(x)L^{-1/2}f(x)=|x_1|^{-1/p}\chi_{B(0,1)}(x)\int_{B(0,1)} \tilde \Gamma(x,y)\, dy \geq c'|x_1|^{-1/p}\chi_{B(0,1)}(x).
    \end{split}
\end{equation*}

{\bf Counterexample 3.}  Now we discuss boundedness of $V^{1/2}L^{-1/2}$ on $L^1(\mathbb R^d)$. Suppose  $d\geq 3$. There is $V\in L^1_{\rm loc} (\mathbb R^d)$, $V\geq 0$,  such that $V^{1/2}L^{-1/2}$ is not bounded on $L^1(\mathbb R^d)$.

To see this, let $V(x)=(1+|x|)^{-2}\big\{ \ln (4+|x|)\big\}^{-2}.$ It is not difficult to check that  $V$ is Green bounded, that is,
\begin{equation}\label{Green}
\sup_{x\in\mathbb R^d} \int_{\mathbb R^d} V(y)|x-y|^{2-d}\, dy<\infty.
\end{equation}
It is well known (see, e.g. \cite{Semenov}) that   \eqref{Green} is equivalent to  the upper and lower Gaussian bounds \eqref{Gaussian-bounds} for the kernel $k_t(x,y)$ of the semigroup $e^{-t(-\Delta +V)}$. Hence \eqref{similar} holds. Now let us consider  $f(x)=\chi_{B(0,1)}(x)$. Then $V^{1/2}L^{-1/2}f$ is not integrable at infinity, since for $|x|>100$, by \eqref{similar}, there is $c>0$ such that
\begin{equation}
    \begin{split}
        V^{1/2} L^{-1/2} f(x)\geq c (1+|x|)^{-1} \big\{ \ln (4+|x|)\big\}^{-1}|x|^{1-d}.
    \end{split}
\end{equation}

\



\end{document}